\documentclass[12pt]{article}
\usepackage{amssymb}
\usepackage{amsmath}
\usepackage{amsthm}
\usepackage{color}
\usepackage{comment}
\usepackage{enumitem}					
\usepackage{geometry}					
\usepackage{graphicx}

\let\originalleft\left
\let\originalright\right
\renewcommand{\left}{\mathopen{}\mathclose\bgroup\originalleft}
\renewcommand{\right}{\aftergroup\egroup\originalright}

\newcommand{\doroverline}[2]{\overline{#1#2}}
\newcommand{\roverline}[1]{\mathpalette\doroverline{#1}}
\newcommand{\lineSeg}[2]{\roverline{#1 #2}}

\newlist{romanlist}{enumerate}{3}
\setlist[romanlist]{label=\roman*),ref=(\roman*)}

\begin{document}

\newcommand{\cR}{\mathcal{R}}
\newcommand{\cU}{\mathcal{U}}
\newcommand{\cV}{\mathcal{V}}
\newcommand{\ee}{\varepsilon}
\newcommand{\rD}{{\rm D}}

\newcommand{\removableFootnote}[1]{\footnote{#1}}

\newtheorem{theorem}{Theorem}[section]
\newtheorem{corollary}[theorem]{Corollary}
\newtheorem{lemma}[theorem]{Lemma}
\newtheorem{proposition}[theorem]{Proposition}

\theoremstyle{definition}
\newtheorem{definition}{Definition}[section]
\newtheorem{example}[definition]{Example}

\theoremstyle{remark}
\newtheorem{remark}{Remark}[section]



\title{
Robust Devaney chaos in the two-dimensional border-collision normal form.
}
\author{
I.~Ghosh and
D.J.W.~Simpson\\\\
School of Fundamental Sciences\\
Massey University\\
Palmerston North\\
New Zealand
}
\maketitle


\begin{abstract}

The collection of all non-degenerate, continuous, two-piece, piecewise-linear maps on $\mathbb{R}^2$
can be reduced to a four-parameter family known as the two-dimensional border-collision normal form.
We prove that throughout an open region of parameter space this family has an attractor satisfying Devaney's definition of chaos.
This strengthens existing results on the robustness of chaos in piecewise-linear maps.
We further show that the stable manifold of a saddle fixed point, despite being a one-dimensional object,
densely fills an open region containing the attractor.
Finally we identify a heteroclinic bifurcation, not described previously,
at which the attractor undergoes a crisis and may be destroyed.

\end{abstract}

\section{Introduction}
\label{sec:intro}
\setcounter{equation}{0}

Robust chaos refers to the phenomenon that a family of dynamical systems
has a chaotic attractor throughout an open region of parameter space \cite{ZeSp12}.
This does not occur for generic families of smooth one-dimensional maps as these have dense windows of periodicity \cite{Va10},
but is typical for systems with sufficiently many dimensions \cite{GoGo18,GoKa21},
a well-known example being the Lorenz system \cite{GuWi79,Tu99}.
One can impose further requirements on the robustness,
such as that the attractor
varies continuously with respect to Hausdorff distance \cite{GlSi20b}
or Lebesgue measure \cite{AlPu17}.

Robust chaos was popularised by Banerjee, Yorke, and Grebogi in \cite{BaYo98}
where they studied the four-parameter family of maps
\begin{equation}
f_\xi(x,y) = \begin{cases}
\begin{bmatrix} \tau_L x + y + 1 \\ -\delta_L x \end{bmatrix}, & x \le 0, \\
\begin{bmatrix} \tau_R x + y + 1 \\ -\delta_R x \end{bmatrix}, & x \ge 0,
\end{cases}
\label{eq:f}
\end{equation}
where, for convenience, we write
\begin{equation}
\xi = \left( \tau_L, \delta_L, \tau_R, \delta_R \right).
\nonumber
\end{equation}
This family is the two-dimensional border-collision normal form
except the border-collision bifurcation parameter, usually denoted $\mu$, has been scaled to $1$.
It describes the dynamics created in {\em border-collision bifurcations} --- where a fixed point of a piecewise-smooth map
collides with a switching manifold \cite{DiBu08,Si16}.
In this context, piecewise-linear maps have been used to explain bifurcations in diverse applications such as
power converters \cite{ZhMo08b}
and mechanical systems with stick/slip friction \cite{SzOs08}.
The family \eqref{eq:f} is a normal form in the sense that any non-degenerate, continuous, two-piece, piecewise-linear map
on $\mathbb{R}^2$ can be converted to \eqref{eq:f} via an affine change of coordinates \cite{NuYo92}.

Banerjee {\em et.~al.}~\cite{BaYo98} identified an open parameter region $\Phi_{\rm BYG} \subset \mathbb{R}^4$,
defined below, throughout which $f_\xi$ exhibits robust chaos.
Recently in \cite{GhSi21} we used renormalisation to partition a region slightly larger than $\Phi_{\rm BYG}$
into subregions $\cR_0, \cR_1, \cR_2, \ldots$.
We showed that if $\xi \in \cR_n$ for some $n \ge 0$,
then $f_\xi$ has a chaotic attractor with $2^n$ connected components.
Fig.~\ref{fig:zd_Lambda} shows the attractor for typical $\xi \in \cR_0$ where it has one connected component
and is the closure of the unstable manifold of a saddle fixed point $X$.
We were further able to show that if $\xi \in \cR_n$ for some $n \ge 1$,
then $f_\xi^{2^n}$ is conjugate to $f_\eta$ for some $\eta \in \cR_0$.
In this way anything one can prove about the dynamics in $\cR_0$ immediately extends to every $\cR_n$ with $n \ge 1$.
For this reason it is helpful to better understand the dynamics in $\cR_0$.

\begin{figure}[t!]
\begin{center}
\includegraphics[width=8cm]{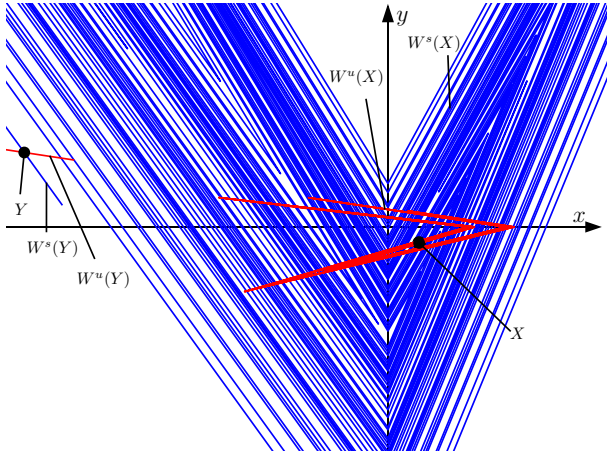}
\caption{
A phase portrait of \eqref{eq:f} with $\xi = \left( 1.5, 0.2, -2, 0.5 \right)$
which corresponds a point in $\cR_0$, see Fig.~\ref{fig:zd_Slices}(a).
There are two saddle fixed points, $X$ and $Y$.
The figure shows numerical computations of the stable (blue) and unstable (red) manifolds of $X$.
The closure of the unstable manifold of $X$, denoted $\Lambda$, is a chaotic attractor;
by Theorem \ref{th:DevaneyChaos} the map exhibits Devaney chaos on $\Lambda$.
The stable manifold $W^s(X)$ can be `grown' by iterating a certain line segment $\lineSeg{f_\xi(V)}{f_\xi^{-1}(V)}$
backwards under $f_\xi$, see \S\ref{sec:XY}, and this figure shows this line segment and its subsequent eight preimages.
By Theorem \ref{th:stableManifoldDense}, $W^s(X)$ is dense in a region containing $\Lambda$.
\label{fig:zd_Lambda}
} 
\end{center}
\end{figure}

Already it is known that $f_\xi$ has a chaotic attractor for all $\xi \in \cR_0$,
but only in the sense of a positive Lyapunov exponent.
In this paper we extend a result of \cite{GlSi21} and show that throughout
a relatively large subset of $\cR_0$ the attractor satisfies Devaney's definition of chaos:
transitivity, dense periodic orbits, and sensitive dependence on initial conditions \cite{De89}.
By the renormalisation discussed above, these properties also hold in the corresponding subsets of $\cR_n$ with $n \ge 1$.

The key objects that we employ to verify Devaney chaos are {\em invariant expanding cones}.
These allow us to obtain a lower bound for the rate at which line segments grow
under iteration by either piece of $f_\xi$.
Indeed such cones have been used to verify
transitivity in other families of piecewise-linear maps \cite{Mi80,Sa99b}.

Below we also show that the stable manifold of $X$ is dense in an open region containing the attractor.
This is illustrated in Fig.~\ref{fig:zd_Lambda} where we have numerically grown $W^s(X)$ outwards for eight iterations
beyond its first kink with the switching manifold, $x=0$.
Already we can see the manifold is leaving only small gaps;
by iterating further we have observed that the size of the gaps steadily reduces further.
Thus $W^s(X)$ behaves like a two-dimensional object, when really it is one-dimensional.
Such a manifold is termed a {\em blender} and for smooth invertible maps only occurs for maps that are
at least three-dimensional \cite{HiKr18}.
Blenders are a central feature of hetero-dimensional cycles and useful for explaining
the breakdown of uniform hyperbolicity \cite{BoDi05}.

The remainder of this paper is organised as follows.
The main results are presented in \S\ref{sec:results}.
Here we also show that the chaotic attractor can persist
beyond $\Phi_{\rm BYG}$ and even outside $\cR_1$.
It appears that if the attractor is not destroyed at the curved boundary of $\Phi_{\rm BYG}$,
then it is destroyed in a heteroclinic bifurcation where the unstable manifold of $X$
develops an intersection with the stable manifold of a period-three solution.

In \S\ref{sec:XY} we introduce essential features of the phase space of $f_\xi$,
then in \S\ref{sec:invariantCone} identify invariant expanding cones on the tangent space of $f_\xi$.
We put these together in \S\ref{sec:lineSegments} to prove statements about
how line segments map under $f_\xi$.
In \S\ref{sec:fInverse} we establish related results for the inverse of $f_\xi$,
then in \S\ref{sec:Devaney} prove transitivity and the denseness of periodic solutions.
Sensitive dependence follows immediately as it is actually a redundant aspect of Devaney's definition \cite{BaBr92}.
Concluding remarks are provided in \S\ref{sec:conc}.

\section{Main results}
\label{sec:results}
\setcounter{equation}{0}

Let
\begin{equation}
\Phi = \left\{ \xi \in \mathbb{R}^4 \,\middle|\, \tau_L > \delta_L + 1, \,\delta_L > 0, \,\tau_R < -(\delta_R + 1), \,\delta_R > 0 \right\}.
\label{eq:saddleSaddleRegion}
\end{equation}
For all $\xi \in \Phi$, the map $f_\xi$ has two fixed points
\begin{align}
X &= \left( \frac{1}{\delta_R + 1 - \tau_R}, \frac{-\delta_R}{\delta_R + 1 - \tau_R} \right), \label{eq:X} \\
Y &= \left( \frac{-1}{\tau_L - \delta_L - 1}, \frac{\delta_L}{\tau_L - \delta_L - 1} \right), \label{eq:Y} 
\end{align}
where $X$ belongs to the right half-plane ($x > 0$)
and $Y$ belongs to the left half-plane ($x < 0$), Fig.~\ref{fig:zd_Lambda}.
Let
\begin{align}
A_L &= \begin{bmatrix} \tau_L & 1 \\ -\delta_L & 0 \end{bmatrix}, &
A_R &= \begin{bmatrix} \tau_R & 1 \\ -\delta_R & 0 \end{bmatrix},
\nonumber
\end{align}
denote the Jacobian matrices of $f_\xi$ in the left and right half-planes.
For all $\xi \in \Phi$,
the eigenvalues of $A_L$ satisfy $0 < \lambda_L^s < 1 < \lambda_L^u$
and the eigenvalues of $A_R$ satisfy $\lambda_R^u < -1 < \lambda_R^s < 0$.
Both $X$ and $Y$ are saddles --- their stable and unstable manifolds are one-dimensional.
In particular, let
\begin{equation}
\Lambda = {\rm cl}(W^u(X))
\label{eq:Lambda}
\end{equation}
denote the closure of the unstable manifold of $X$.

\subsection{A division of parameter space}

The robust chaos parameter region of Banerjee, Yorke, and Grebogi \cite{BaYo98} is the set
\begin{equation}
\Phi_{\rm BYG} = \left\{ \xi \in \Phi \,\middle|\, \phi(\xi) > 0 \right\},
\label{eq:BYGRegion}
\end{equation}
where
\begin{equation}
\phi(\xi) = \delta_R - \left( \tau_R + \delta_L + \delta_R - (1+\tau_R) \lambda_L^u \right) \lambda_L^u.
\label{eq:phi}
\end{equation}
Banerjee {\em et.~al.}~observed that $\Lambda$ may be a chaotic attractor
that can be destroyed when $\phi(\xi) = 0$, as this is
where the stable and unstable manifolds of $Y$ develop an intersection.
This bifurcation is a {\em homoclinic corner} \cite{Si16b}
analogous to a `first' homoclinic tangency for smooth maps \cite{PaTa93}.

In \cite{GhSi21} we introduced the renormalisation operator $g : \Phi \to \Phi$ defined by
\begin{equation}
g(\xi) = \big( \tau_R^2 - 2 \delta_R, \delta_R^2, \tau_L \tau_R - \delta_L - \delta_R, \delta_L \delta_R \big).
\nonumber
\end{equation}
For all $n \ge 0$, let
\begin{equation}
\cR_n = \left\{ \xi \in \Phi \,\middle|\, \phi \left( g^n(\xi) \right) > 0,\, \phi \left( g^{n+1}(\xi) \right) \le 0 \right\}.
\label{eq:Rn}
\end{equation}
As shown in \cite{GhSi21}, the sets $\cR_n$ are non-empty, mutually disjoint, and converge to $\{ \xi^* \}$ as $n \to \infty$,
where $\xi^* = (1,0,-1,0)$ is a fixed point of $g$ that lies on the boundary of $\Phi$.
For any two-dimensional slice of parameter space defined by fixing the values of $\delta_L > 0$ and $\delta_R > 0$,
only finitely many $\cR_n$ are visible.
Moreover, $\cR_0$ and $\cR_1$ are by far the largest regions.
Indeed in the example slices shown in Fig.~\ref{fig:zd_Slices}
these are the only $\cR_n$ that are visible.

\begin{figure}[b!]
\begin{center}
\setlength{\unitlength}{1cm}
\begin{picture}(17,8.15)
\put(0,0){\includegraphics[width=8cm]{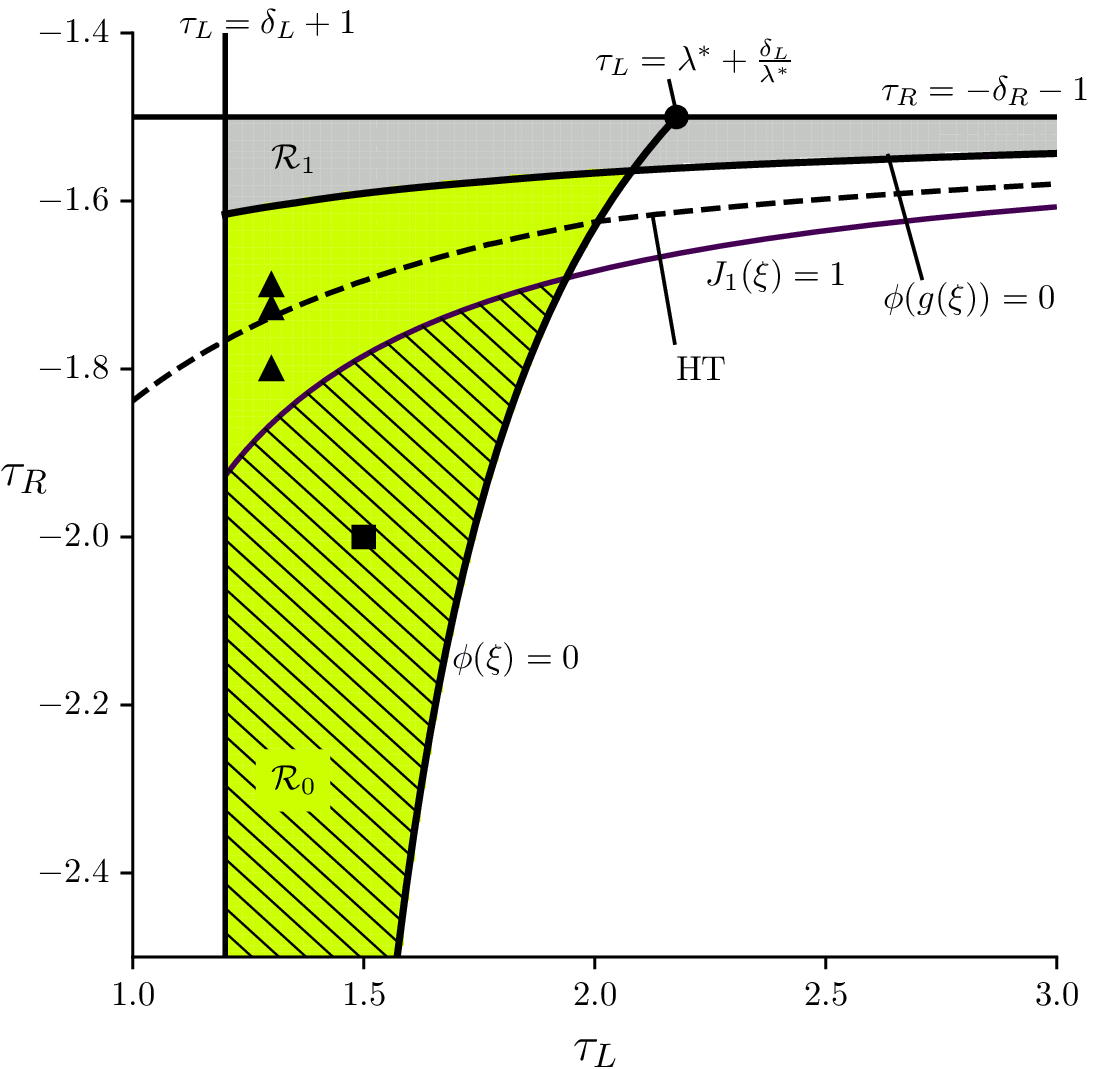}}
\put(9,0){\includegraphics[width=8cm]{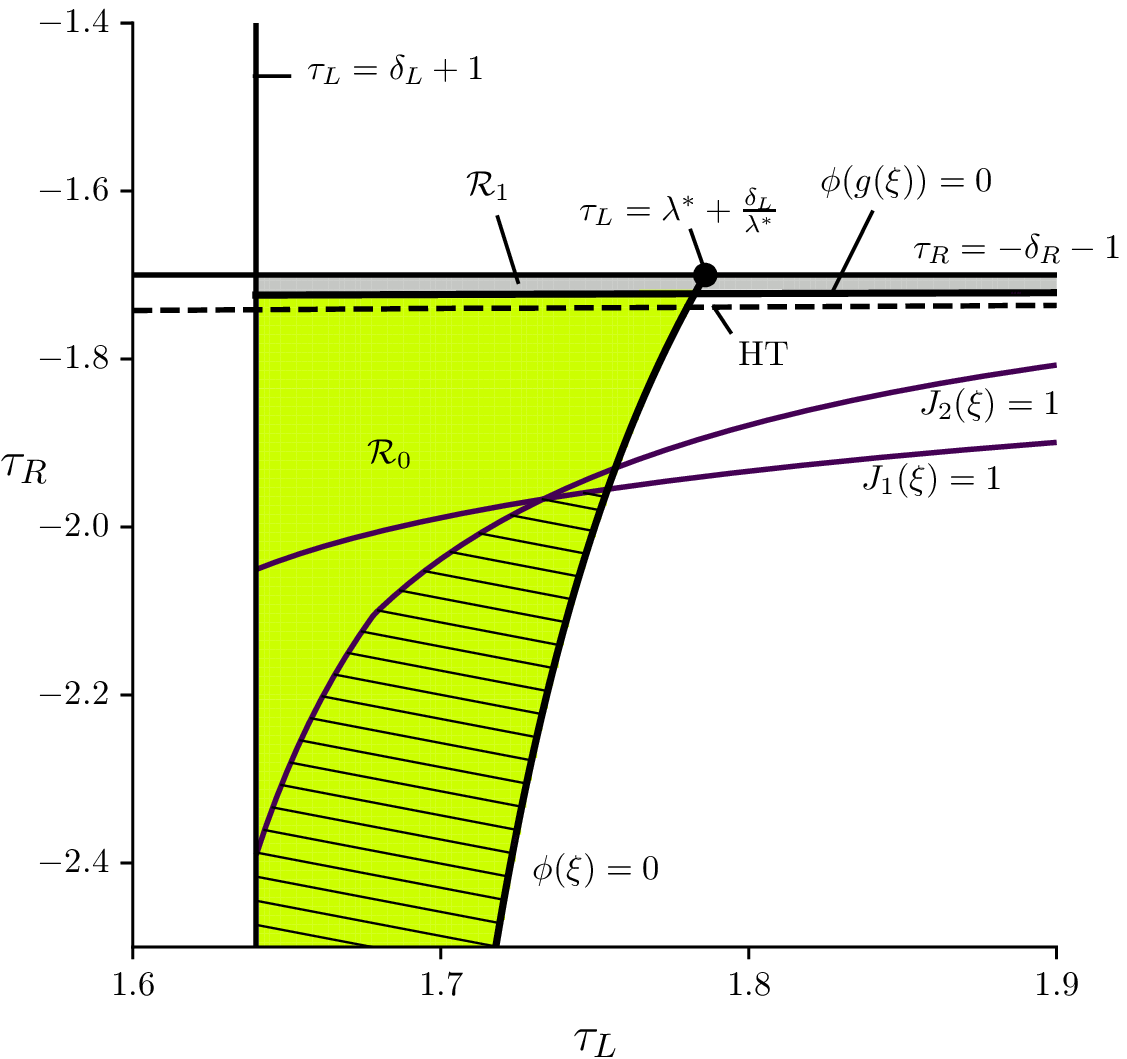}}
\put(2.4,8){\small {\bf a)}~~$\delta_L = 0.2$, $\delta_R = 0.5$}
\put(11.4,8){\small {\bf b)}~~$\delta_L = 0.64$, $\delta_R = 0.7$}
\end{picture}
\caption{
Two-dimensional slices of parameter space defined by fixing the values of $\delta_L$ and $\delta_R$ as indicated.
The regions $\cR_0$ and $\cR_1$ are coloured yellow and grey respectively.
The curved boundary $\phi(\xi) = 0$ intersects the horizontal boundary $\tau_R = -\delta_R -1$
at $\tau_L = \lambda^* + \frac{\delta_L}{\lambda*}$, where $\lambda^*$ is the largest root of the quadratic
$-\delta_R \lambda^2 + (1-\delta_L) \lambda + \delta_R$, see \cite{GhSi21}.
The striped regions are where Theorems \ref{th:stableManifoldDense} and \ref{th:DevaneyChaos} apply.
The curves labelled HT are heteroclinic bifurcations discussed in \S\ref{sub:HT}.
In panel (a) the black square indicates the parameter values of Fig.~\ref{fig:zd_Lambda};
the black triangles indicate the parameter values of Fig.~\ref{fig:zd_HTA}.
\label{fig:zd_Slices}
} 
\end{center}
\end{figure}

We now state the main results.
These are proved in later sections and involve the functions
\begin{align}
J_1(\xi) &= \frac{\lambda_L^u \lambda_R^{u^2}}{\lambda_L^u + |\lambda_R^u|},
\label{eq:a1} \\
J_2(\xi) &= {\rm max} \left\{ \lambda_L^s, \frac{\sqrt{2} \lambda_L^s}{\lambda_L^s + 1} \right\} + 
{\rm max} \left\{ |\lambda_R^s|, \frac{\sqrt{2} |\lambda_R^s|}{|\lambda_R^s| + 1} \right\}.
\label{eq:a2strong}
\end{align}
Since the eigenvalues $\lambda_R^s$ and $\lambda_R^u$ are negative we have used absolute values
to make the signs of the various quantities readily apparent.
The results here tell about the dynamics of $f_\xi$ in $\cR_0$
because $J_1(\xi) > 1$ implies $\xi \in \cR_0$ for any $\xi \in \Phi_{\rm BYG}$, see Lemma \ref{le:Zexists}.

\begin{theorem}
Let $\xi \in \Phi_{\rm BYG}$ and suppose $J_1(\xi) > 1$ and $\lambda_L^s + |\lambda_R^s| < 1$.
Then $W^s(X)$ is dense in a triangular region containing $\Lambda$.
\label{th:stableManifoldDense}
\end{theorem}

\begin{theorem}
Let $\xi \in \Phi_{\rm BYG}$ and suppose $J_1(\xi) > 1$ and $J_2(\xi) < 1$.
Then $f_\xi$ is chaotic in the sense of Devaney on $\Lambda$.
\label{th:DevaneyChaos}
\end{theorem}

Notice $J_2(\xi) < 1$ is a stronger condition than $\lambda_L^s + |\lambda_R^s| < 1$.
The striped regions in Fig.~\ref{fig:zd_Slices} show where the conditions of Theorem \ref{th:DevaneyChaos} hold
for the given fixed values of $\delta_L$ and $\delta_R$.
As mentioned above these are necessarily subsets of $\cR_0$.
In panel (a) the condition $J_2(\xi) < 1$ does not influence the boundary of the striped region
because this condition is satisfied automatically when the values of $\delta_L$ and $\delta_R$ are sufficiently small.
In panel (b) we have chosen values of $\delta_L$ and $\delta_R$ that highlight
the nonsmoothness of the constraint $J_2(\xi) < 1$.

The conditions $J_1(\xi) > 1$ and $J_2(\xi) < 1$ arise quite naturally.
First $J_1(\xi) > 1$ ensures the second iterate $f_\xi^2$ is expanding in some directions,
whilst $J_2(\xi) < 1$ ensures the inverse $f_\xi^{-1}$ is expanding in some other directions.
Below we use these expansion properties to verify Devaney chaos.
Nevertheless, numerical explorations suggest that these conditions are sub-optimal and that
$\Lambda$ in fact has Devaney chaos for all $\xi \in \cR_0$.
A proof of this remains for future work.

\subsection{A heteroclinic bifurcation involving $\Lambda$}
\label{sub:HT}

For the most part $\Lambda$ appears to vary continuously as $\xi$ is varied continuously within $\cR_0$,
but jumps in size when it develops intersections with the stable manifold of a period-three cycle.
This is illustrated in Fig.~\ref{fig:zd_HTA}.
By increasing the value of $\tau_R$ an intersection first occurs at $\tau_R \approx -1.727455$.
This heteroclinic bifurcation is an example of a {\em crisis} \cite{GrOt83}.

\begin{figure}[h!]
\begin{center}
\setlength{\unitlength}{1cm}
\begin{picture}(17,5.1)
\put(0,0){\includegraphics[width=5.2cm]{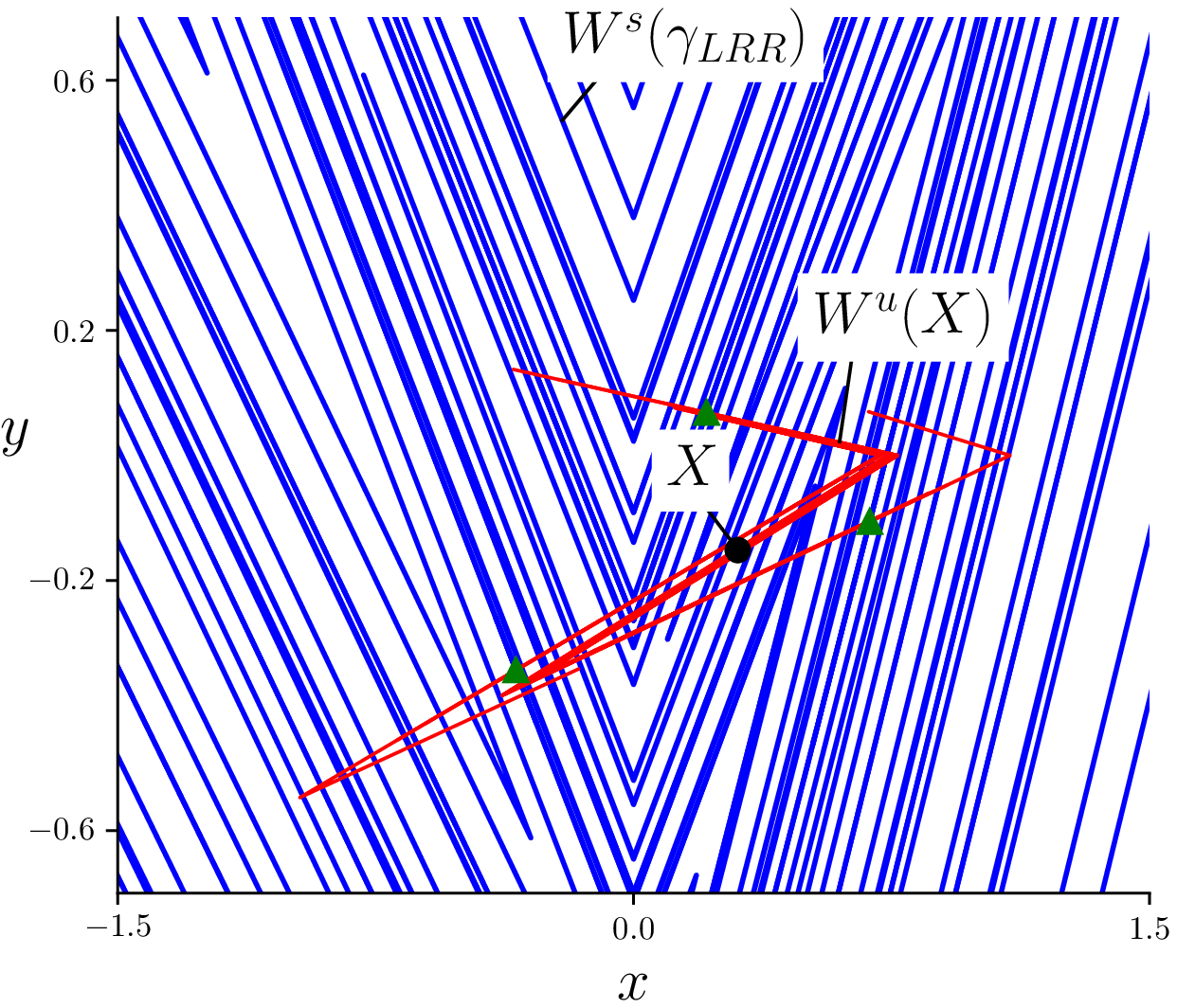}}
\put(5.9,0){\includegraphics[width=5.2cm]{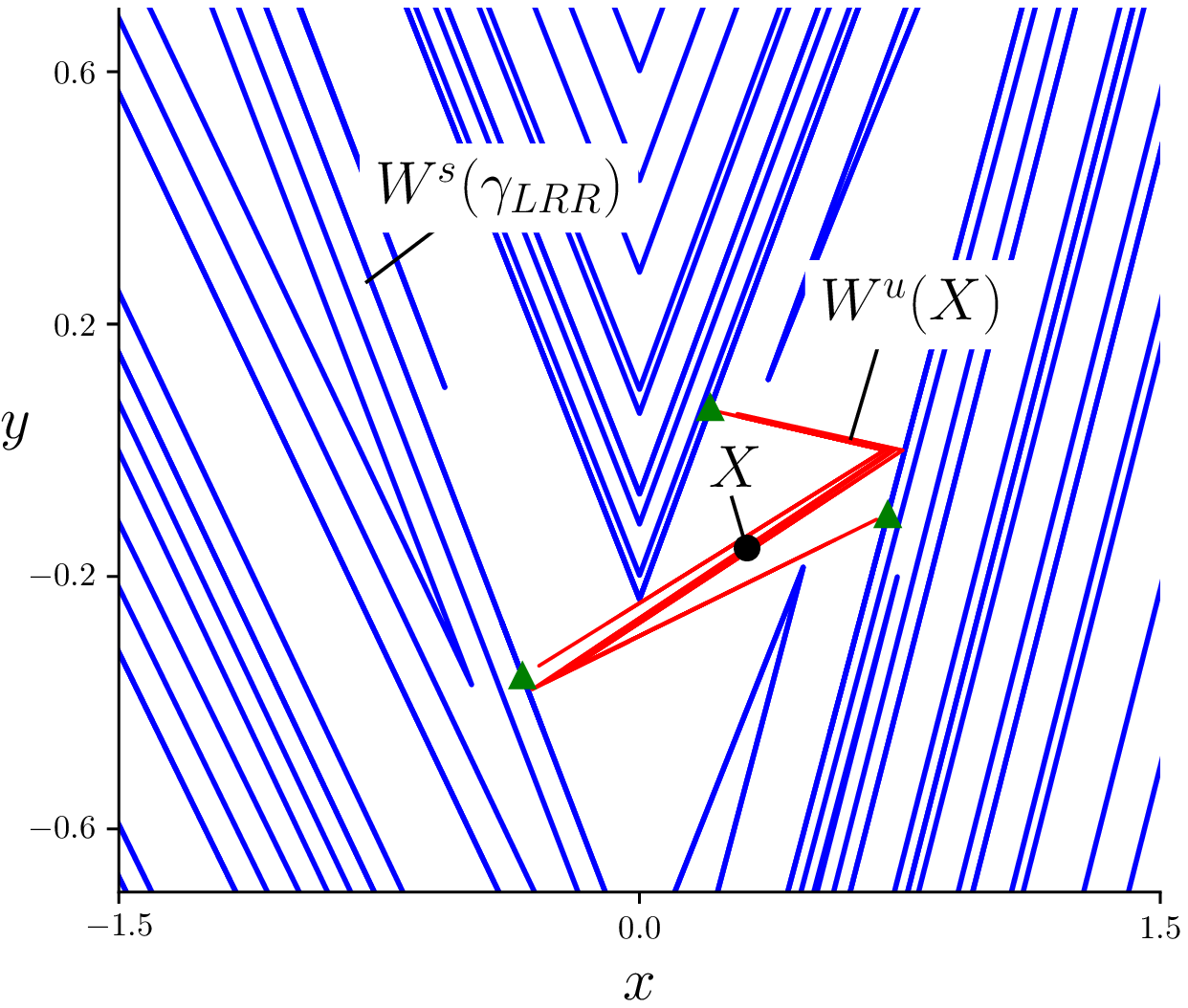}}
\put(11.8,0){\includegraphics[width=5.2cm]{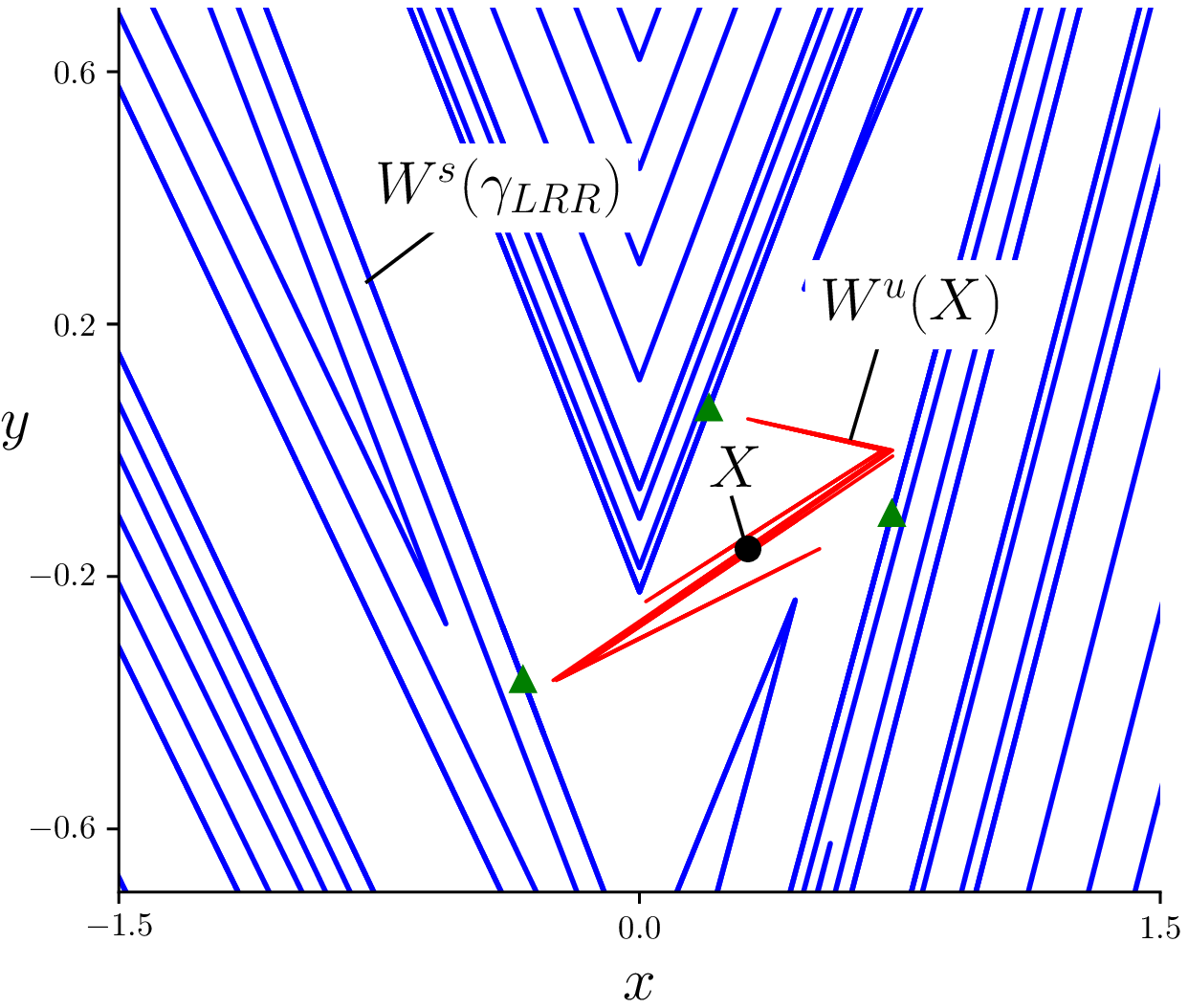}}
\put(1.7,4.95){\small {\bf a)}~~$\tau_R = -1.7$}
\put(7,4.95){\small {\bf b)}~~$\tau_R = -1.727455$}
\put(13.5,4.95){\small {\bf c)}~~$\tau_R = -1.8$}
\end{picture}
\caption{
Phase portraits of \eqref{eq:f} with $\delta_L = 0.2$, $\delta_R = 0.5$, $\tau_L = 1.3$,
and three different values of $\tau_R$ as indicated.
These parameter combinations are shown in Fig.~\ref{fig:zd_Slices}.
In each plot the three green triangles are the points of a saddle-type $LRR$-cycle, denoted $\gamma_{LRR}$.
\label{fig:zd_HTA}
} 
\end{center}
\end{figure}

Numerically we have computed curves in Fig.~\ref{fig:zd_Slices} (labelled HT) where the heteroclinic bifurcation occurs.
This was a computationally intensive task because, unlike the simple corner intersections
associated with $\phi(\xi) = 0$ for example
that involve the first couple of linear pieces of $W^s(Y)$ and $W^u(Y)$ as they emanate outwards from $Y$,
the corner intersections associated with the heteroclinic bifurcation
involve the limiting outer-most part of $W^u(X)$ as is it grown outwards from $X$ indefinitely.
Curves of such bifurcations were considered in \cite{Os06} where it was
argued that other bifurcations may contribute to the crisis.
It remains to be seen whether or not such bifurcations are present in our piecewise-linear setting.

The attractor $\Lambda$ appears to persist to the right of $\phi(\xi) = 0$
in the narrow strip bounded above by $\phi(g(\xi)) = 0$
and bounded below by the heteroclinic bifurcation curve.
From a point within this strip,
if we cross $\phi(g(\xi)) = 0$ we enter $\cR_1$ and $\Lambda$ is replaced by an attractor with two connected components,
while if we cross the heteroclinic bifurcation forward orbits of points near $X$ are able to diverge
so $W^u(X)$ suddenly becomes unbounded and the attractor is destroyed.
Numerical investigations suggest this occurs for any values of $\delta_L > 0$ and $\delta_R > 0$,
that is, it is always the stable manifold of the $LRR$-cycle (period-three solution with symbolic itinerary $LRR$)
that is responsible for the crisis.

\subsection{Additional notation}

Here we collectively introduce additional notation and conventions used regularly in subsequent sections.

We write $\| \cdot \|$ for the Euclidean norm on $\mathbb{R}^2$.
We use subscripts to denote the components of a point, e.g.~$P = (P_1,P_2) \in \mathbb{R}^2$.
We write $B_\ee(P) = \left\{ Q \in \mathbb{R}^2 \,\middle|\, \| P - Q \| < \ee \right\}$
for the ball of radius $\ee > 0$ centred at a point $P \in \mathbb{R}^2$.

Given $P, Q \in \mathbb{R}^2$,
the line segment connecting $P$ and $Q$, and including $P$ and $Q$, is denoted $\lineSeg{P}{Q}$.
The {\em length} of a line segment $\lineSeg{P}{Q}$ is the distance between its endpoints
and denoted $|\lineSeg{P}{Q}|$.

The switching manifold of $f_\xi$ is the set
$\Sigma = \left\{ (0,y) \,\middle|\, y \in \mathbb{R} \right\}$ --- the $y$-axis.
For any $\xi \in \mathbb{R}^4$, its image $f_\xi(\Sigma)$ is the $x$-axis.

\section{The stable and unstable manifolds of the fixed points}
\label{sec:XY}
\setcounter{equation}{0}

Since $f_\xi$ is piecewise-linear,
the stable and unstable manifolds of the fixed points $X$ and $Y$ are themselves piecewise-linear.
As they emanate from the fixed point they coincide with the corresponding stable or unstable subspace.

Let us first consider the unstable manifold of $Y$, denoted $W^u(Y)$, which corresponds to the
unstable eigenvalue $\lambda_L^u$.
As $W^u(Y)$ emanates from $Y$ it coincides with the unstable subspace, $E^u(Y)$,
and has two dynamically independent branches because $\lambda_L^u$ is positive.
The first kink of the right branch of $W^u(Y)$, as we follow it outwards from $Y$,
occurs at the point
\begin{equation}
D = \left( \frac{1}{1 - \lambda_L^s}, 0 \right),
\label{eq:D}
\end{equation}
see Fig.~\ref{fig:zd_fOmega}.
For all $\xi \in \Phi$, $D$ lies to the right of $(1,0)$, whereas $f_\xi(D)$ lies in the left half-plane.
Let $U$ denote the intersection of $\lineSeg{D}{f_\xi(D)}$ with $\Sigma$.
The right branch of $W^u(Y)$ contains the line segments $\lineSeg{Y}{D}$ and $\lineSeg{D}{f_\xi(D)}$
and can be `grown' by iterating $\lineSeg{D}{f_\xi(D)}$ under $f_\xi$
because this line segment, minus one of its endpoints,
is a {\em fundamental domain} \cite{PaTa93} for this branch of the manifold.

\begin{figure}[b!]
\begin{center}
\includegraphics[width=8cm]{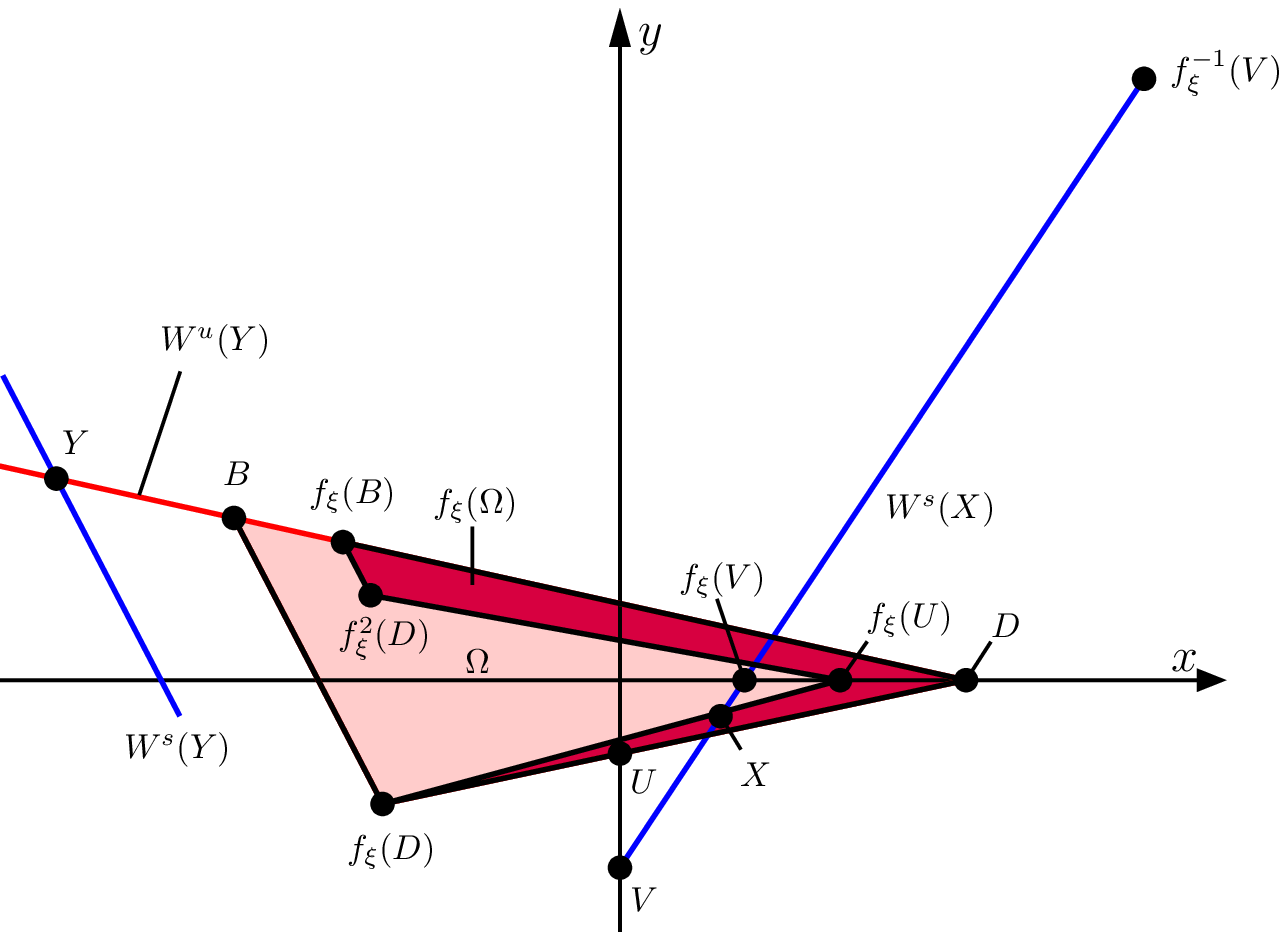}
\caption{
A sketch of the forward invariant region $\Omega$
that is constructed using kinks of the unstable manifold of $Y$ (red).
We also show its image $f_\xi(\Omega)$
and parts of the stable manifolds of $X$ and $Y$ (blue).
\label{fig:zd_fOmega}
} 
\end{center}
\end{figure}

As in \cite{GlSi21}, let $B$ denote the intersection of
$\lineSeg{Y}{D}$ with the line through $f_\xi(D)$ and parallel to the stable subspace $E^s(Y)$, see Fig.~\ref{fig:zd_fOmega}.
Then let $\Omega$ be the filled triangle with vertices $D$, $f_\xi(D)$, and $B$.
The following result characterises the image of $\Omega$ under $f_\xi$.
An immediate consequence is that
$\Omega$ is forward invariant under $f_\xi$ for all $\xi \in \Phi_{\rm BYG}$.
Essentially this is the case because $\xi \in \Phi_{\rm BYG}$ implies $f_\xi(D)$ lies to the right of $E^s(Y)$;
refer to \cite{GlSi21} for a proof.

\begin{lemma}
Let $\xi \in \Phi_{\rm BYG}$.
Then $f_\xi(\Omega) \subset \Omega$ and is the polygon with vertices
$D$, $f_\xi(D)$, $f_\xi(U)$, $f_\xi^2(D)$, and $f_\xi(B)$.
\label{le:fOmega}
\end{lemma}

We now consider the stable manifold of $X$, $W^s(X)$.
This manifold has only one dynamically independent branch
because the corresponding eigenvalue $\lambda_R^s$ is negative.
The first kink of $W^s(X)$, as we follow it outwards from $X$, occurs at the point
\begin{equation}
V = \left( 0, \frac{-\lambda_R^u}{\lambda_R^u - 1} \right),
\label{eq:V}
\end{equation}
see again Fig.~\ref{fig:zd_fOmega}.
The stable manifold $W^s(X)$ contains $\lineSeg{V}{f_\xi^{-1}(V)} \subset E^s(X)$
and can be grown by iterating $\lineSeg{f_\xi(V)}{f_\xi^{-1}(V)}$ under $f_\xi^{-1}$.
This is because $\lineSeg{f_\xi(V)}{f_\xi^{-1}(V)}$, minus one of its endpoints, is a fundamental domain for $W^s(X)$.

The next result clarifies how $\lineSeg{V}{f_\xi^{-1}(V)}$
intersects the region $f_\xi(\Omega)$
and is used below to prove Lemma \ref{le:intersectionWithBothAxes}.
Specifically Lemma \ref{le:U2V2} implies $f_\xi(V)$ lies to the left of $f_\xi(U)$;
consequently $\lineSeg{V}{f_\xi^{-1}(V)}$ cuts $f_\xi(\Omega)$ into three pieces.

\begin{lemma}
Let $\xi \in \Phi$.
If $\lambda_L^s + |\lambda_R^s| < 1$ then $U_2 > V_2$.
\label{le:U2V2}
\end{lemma}

\begin{proof}
By iterating \eqref{eq:D} under $f_\xi$ we find that $\lineSeg{D}{f_\xi(D)}$
has slope $\frac{\delta_R}{\lambda_L^s - \tau_R}$ and consequently
\begin{equation}
U = \left( 0, \frac{-\delta_R}{\left( \lambda_L^s - \tau_R \right) \left( 1 - \lambda_L^s \right)} \right).
\label{eq:U}
\end{equation}
By subtracting the second component of \eqref{eq:V} from that of \eqref{eq:U}
and factorising, we obtain
\begin{equation}
U_2 - V_2 =
\frac{\left| \lambda_R^u \right| \left( 1 - \lambda_L^s + \lambda_R^s \right) \left( \lambda_L^s - \lambda_R^u \right)}
{\left( 1 - \lambda_L^s \right) \left( 1 - \lambda_R^u \right) \left( \lambda_L^s - \tau_R \right)}.
\label{eq:U2minusV2}
\end{equation}
If $\lambda_L^s + |\lambda_R^s| < 1$ then each bracketed factor in \eqref{eq:U2minusV2} has a positive value,
thus $U_2 > V_2$.
\end{proof}

\section{Invariant expanding cones}
\label{sec:invariantCone}
\setcounter{equation}{0}

In this section we construct invariant expanding cones for the Jacobian matrices $A_L$ and $A_R$.

\begin{definition}
A nonempty set $C \subset \mathbb{R}^2$ is a {\em cone}
if $t v \in C$ for all $v \in C$ and $t \in \mathbb{R}$.
\label{df:cone}
\end{definition}

\begin{definition}
Let $A$ be a real-valued $2 \times 2$ matrix.
A cone $C \subset \mathbb{R}^2$ is {\em invariant} if $A v \in C$ for all $v \in C$.
The cone is {\em expanding} if there exists an expansion factor $c > 1$ such that
$\left\| A v \right\| \ge c \| v \|$ for all $v \in C$.
\label{df:iec}
\end{definition}

Given an interval $K \subset \mathbb{R}$, the set
\begin{equation}
\Psi_K = \left\{ t \begin{bmatrix} 1 \\ m \end{bmatrix} \,\middle|\, m \in K, t \in \mathbb{R} \right\},
\label{eq:PsiK}
\end{equation}
is a cone.
The vector $v = \begin{bmatrix} 1 \\ -\lambda_L^s \end{bmatrix}$ is an eigenvector of $A_L$ corresponding to the eigenvalue $\lambda_L^u$.
Thus $\left\| A v \right\| = \lambda_L^u \| v \|$.
It is a simple exercise to show that this equation holds with $v = \begin{bmatrix} 1 \\ m \end{bmatrix}$
for exactly one other value of $m \in \mathbb{R}$, namely
\begin{equation}
m_{\rm crit} = \lambda_L^s + \frac{2 \tau_L}{\lambda_L^{u^2} - 1},
\label{eq:mCrit}
\end{equation}
and notice $m_{\rm crit} > 0$.
It is not difficult to show that $\left\| A v \right\| \ge \lambda_L^u \| v \|$
for all $v = \begin{bmatrix} 1 \\ m \end{bmatrix}$ with $-\lambda_L^s \le m \le m_{\rm crit}$.
The following result generalises this observation.

\begin{lemma}
Suppose $\tau_L > \delta_L + 1$ and $\delta_L > 0$.
Let $K = \left[ -\lambda_L^s, m_{\rm max} \right]$ for some $m_{\rm max} \ge 0$.
\begin{romanlist}
\item
If $m_{\rm max} \le m_{\rm crit}$
then $\Psi_K$ is invariant and expanding for $A_L$ with expansion factor $\lambda_L^u$.
\label{it:AL1}
\item
If $m_{\rm max} \le 1$
then $\Psi_K$ is invariant and expanding for $A_L$ with expansion factor
$c_L = {\rm min} \left\{ \lambda_L^u, \frac{\lambda_L^u + 1}{\sqrt{2}} \right\}$.
\label{it:AL2}
\end{romanlist}
\label{le:AL}
\end{lemma}

\begin{proof}
Choose any $v \in \Psi_K$.
Then $v = t \begin{bmatrix} 1 \\ m \end{bmatrix}$ for some $m \in K$ and $t \in \mathbb{R}$.
If $t = 0$ then $A_L v \in \Psi_K$ trivially,
otherwise the slope 
of $A_L v$ is $G(m) = \frac{-\delta_L}{\tau_L + m}$.
Notice $G \left( -\lambda_L^s \right) = -\lambda_L^s$.
Also $G$ is increasing and negative-valued on $[0,\infty)$.
Therefore $G(m) \in \left[ -\lambda_L^s, 0 \right) \subset K$ and so $A_L v \in \Psi_K$.
That is, $\Psi_K$ is invariant for $A_L$.

For any $c > 1$ and $t \ne 0$, 
\begin{equation}
H(m) = \| A_L v \|^2 - c^2 \| v \|^2
= t^2 \Big( \left( 1 - c^2 \right) m^2 + 2 \tau_L m + \tau_L^2 + \delta_L^2 - c^2 \Big),
\label{eq:ALProof10}
\end{equation}
is a concave down quadratic function of $m$.
Thus, on $K$, $H$ achieves its minimum at an endpoint of $K$.
With $m = \lambda_L^s$ or $m = m_{\rm crit}$ we have
\begin{equation}
\frac{\| A_L v \|^2}{\| v \|^2} = \lambda_L^{u^2}.
\label{eq:ALProof20}
\end{equation}
Thus with $m_{\rm max} \le m_{\rm crit}$ and $c = \lambda_L^u$, then $H(m) \ge 0$ for all $m \in K$.
That is, $\Psi_K$ is expanding with expansion factor $\lambda_L^u$.
With instead $m = 1$ we have
\begin{equation}
\frac{\| A_L v \|^2}{\| v \|^2} = \frac{(\tau_L + 1)^2 + \delta_L^2}{2} > \frac{(\lambda_L^u + 1)^2}{2}.
\label{eq:ALProof21}
\end{equation}
Thus if $m_{\rm max} \le 1$ and $c^2$ is the minimum of \eqref{eq:ALProof20} and the bound in \eqref{eq:ALProof21}
then again $H(m) \ge 0$ for all $m \in K$.
That is, $\Psi_K$ is expanding with expansion factor $c_L$.
\end{proof}

Next we state a result for $A_R$ that is analogous to Lemma \ref{le:AL}\ref{it:AL2}.
This can be obtained in the same way so we do not provide a proof.

\begin{lemma}
Suppose $\tau_R < -\delta_R - 1$ and $\delta_R > 0$.
Let $K = \left[ m_{\rm min}, |\lambda_R^s| \right]$ for some $m_{\rm min} \in [-1,0]$.
The cone $\Psi_K$ is invariant and expanding for $A_R$ with expansion factor
$c_R = {\rm min} \left\{ |\lambda_R^u|, \frac{|\lambda_R^u| + 1}{\sqrt{2}} \right\}$.
\label{le:AR}
\end{lemma}

Next we provide a simple lower bound for the value of $m_{\rm crit}$.
We then use this, and Lemmas \ref{le:AL} and \ref{le:AR},
to construct a cone that is invariant and expanding for both $A_L$ and $A_R$.

\begin{lemma}
Let $\xi \in \Phi_{\rm BYG}$.
Then $m_{\rm crit} > 2 \delta_R$.
\label{le:mCrit}
\end{lemma}

\begin{proof}
Treat $m_{\rm crit}$ as a function of $\tau_L$ and $\delta_L$.
By directly differentiating \eqref{eq:mCrit} we obtain
\begin{equation}
\frac{\partial m_{\rm crit}}{\partial \tau_L} = 
-\left[ \frac{\delta_L}{\lambda_L^{u^2}}
+ \frac{2 \big( \lambda_L^{u^2} + 1 \big)}{\left( \lambda_L^{u^2} - 1 \right)^2}
+ \frac{2 \delta_L \big( 3 \lambda_L^{u^2} - 1 \big)}{\lambda_L^{u^2} \left( \lambda_L^{u^2} - 1 \right)^2} \right]
\frac{\partial \lambda_L^u}{\partial \tau_L}.
\nonumber
\end{equation}
By inspection the quantity in square brackets is positive (because $\delta_L > 0$ and $\lambda_L^u > 1$).
Also $\frac{\partial \lambda_L^u}{\partial \tau_L} > 0$,
therefore, for any fixed $\delta_L > 0$, $m_{\rm crit}$ is a decreasing function of $\tau_L$.

As shown in \cite{GhSi21}, in the $(\tau_L,\tau_R)$-plane the curve $\phi(\xi) = 0$ intersects
$\tau_R = -\delta_R - 1$ at $\tau_L = \lambda^* + \frac{\delta_L}{\lambda^*}$,
where $\lambda^*$ is the larger of the two solutions to
\begin{equation}
-\delta_R \lambda^2 + (1-\delta_L) \lambda + \delta_R = 0,
\label{eq:lambdaStar}
\end{equation}
as indicated in Fig.~\ref{fig:zd_Slices}.
Since the curve $\phi(\xi) = 0$ has positive slope everywhere \cite{GhSi21},
for any $\xi \in \Phi_{\rm BYG}$ the value of $\tau_L$ is less than $\lambda^* + \frac{\delta_L}{\lambda^*}$.
So, since $m_{\rm crit}$ is a decreasing function of $\tau_L$,
$m_{\rm crit} > m_{\rm crit}^*(\delta_L,\delta_R)$ where
\begin{equation}
m_{\rm crit}^*(\delta_L,\delta_R) =
\frac{\delta_L}{\lambda^*} + \frac{2 \left( \lambda^* + \frac{\delta_L}{\lambda^*} \right)}{\lambda^{*^2} - 1},
\nonumber
\end{equation}
is obtained by replacing $\tau_L$ in \eqref{eq:mCrit} with $\lambda^* + \frac{\delta_L}{\lambda^*}$.

It remains to show that $m_{\rm crit}^*(\delta_L,\delta_R) > 2 \delta_R$ for all $0 < \delta_L < 1$ and $\delta_R > 0$
because to have $\xi \in \Phi_{\rm BYG}$ we must have $\delta_L < 1$ \cite{GhSi21}.
First observe $m_{\rm crit}^*(0,\delta_R) = \frac{2 \lambda^*}{\lambda^{*^2} - 1}$
and in view of \eqref{eq:lambdaStar} this reduces to $m_{\rm crit}^*(0,\delta_R) = 2 \delta_R$.
Also
\begin{equation}
\frac{\partial m_{\rm crit}^*}{\partial \delta_L} = 
\frac{1}{\alpha} + \frac{2}{\alpha \left( \alpha^2 - 1 \right)}
- \frac{\frac{\partial \lambda^*}{\partial \delta_L}}{\lambda^{*^2} \left( \lambda^{*^2} - 1 \right)^2}
\left[ \delta_L \left( \lambda^{*^4} + 4 \lambda^{*^2} - 1 \right)
+ 2 \lambda^{*^2} \left( \lambda^{*^2} + 1 \right) \right].
\nonumber
\end{equation}
By inspection the quantity in square brackets is positive (because $\delta_L > 0$ and $\lambda^* > 1$).
It is a simple exercise to show $\frac{\partial \lambda^*}{\partial \delta_L} < 0$,
therefore, for any fixed $\delta_R > 0$, $m_{\rm crit}^*$ is an increasing function of $\delta_L$.
Thus $m_{\rm crit}^*(\delta_L,\delta_R) > m_{\rm crit}^*(0,\delta_R) = 2 \delta_R$, as required.
\end{proof}

\begin{lemma}
Let $\xi \in \Phi_{\rm BYG}$.
With $K = \left[ -\lambda_L^s, |\lambda_R^s| \right]$, the cone $\Psi_K$ is invariant for $A_L$ and $A_R$.
Moreover, it is expanding for $A_L$ with expansion factor $\lambda_L^u$
and expanding for $A_R$ with expansion factor
$c_R = {\rm min} \left\{ |\lambda_R^u|, \frac{|\lambda_R^u| + 1}{\sqrt{2}} \right\}$.
\label{le:K}
\end{lemma}

\begin{proof}
Observe $|\lambda_R^s| < \delta_R < 2 \delta_R < m_{\rm crit}$ using Lemma \ref{le:mCrit}.
Thus the right end-point of $K$ is bounded above by $m_{\rm crit}$,
hence the result for $A_L$ follows from Lemma \ref{le:AL}\ref{it:AL1}.
Also $-1 < -\lambda_L^s$ thus the result for $A_R$ follows from Lemma \ref{le:AR}.
\end{proof}

\section{Expanding line segments}
\label{sec:lineSegments}
\setcounter{equation}{0}

In this section we use invariant expanding cones to study how line segments map under $f_\xi$
and to prove Theorem \ref{th:stableManifoldDense}.

Let $\alpha = \lineSeg{P}{Q} \subset \mathbb{R}^2$ be a line segment and write $Q = P + v$.
The length of $\alpha$ is $|\alpha| = \| v \|$.
We say $\alpha$ {\em crosses} $\Sigma$
if it contains points on both sides of $\Sigma$.
If $\alpha$ does not cross $\Sigma$,
then all points in $\alpha$ map under the same piece of $f_\xi$.
In this case $f_\xi(\alpha)$ is another line segment and 
\begin{equation}
|f_\xi(\alpha)| = \left\| A_J v \right\|,
\label{eq:length}
\end{equation}
with $J = L$ or $J = R$ accordingly.

If $\alpha$ crosses $\Sigma$,
then $f_\xi(\alpha)$ consists of two line segments connected at a point on $f_\xi(\Sigma)$.
The next result (inspired by Lemma 2.1 of \cite{VeGl90})
provides a lower bound on the length of the longer of these two segments.

\begin{lemma}
Let $s$ be the length of the longest line segment in $f_\xi(\alpha)$ and let $c_L, c_R > 0$.
If $\| A_L v \| \ge c_L \| v \|$
and $\| A_R v \| \ge c_R \| v \|$,
then
\begin{equation}
s \ge \frac{c_L c_R}{c_L + c_R} \,|\alpha|.
\label{eq:longestLength}
\end{equation}
\label{le:longestLength}
\end{lemma}

\begin{proof}
Let $\alpha_L$ denote the part of $\alpha$ in the closed left half-plane
and $\alpha_R$ denote the part of $\alpha$ in the closed right half-plane.
Then $|\alpha_L| = t |\alpha|$ and $|\alpha_R| = (1-t) |\alpha|$ for some $t \in [0,1]$.
Observe $|f_\xi(\alpha_L)| = \left\| A_L (t v) \right\| \ge t c_L |\alpha|$
and $|f_\xi(\alpha_R)| = \left\| A_R ((1-t) v) \right\| \ge (1-t) c_R |\alpha|$.
Thus $s \ge {\rm max} \left\{ t c_L, (1-t) c_R \right\} |\alpha|$.
This length is smallest when $t c_L = (1-t) c_R$.
That is, $t = \frac{c_R}{c_L + c_R}$,
giving \eqref{eq:longestLength}.
\end{proof}

We are now able to show that the condition $J_1(\xi) > 1$ implies $\xi \in \cR_0$ for any $\xi \in \Phi_{\rm BYG}$.
This is because if $\xi \in \Phi_{\rm BYG}$ then $\phi(\xi) > 0$,
so if also $\phi(g(\xi)) < 0$ then $\xi \in \cR_0$ by definition, \eqref{eq:Rn}.

\begin{lemma}
Let $\xi \in \Phi$.
If $J_1(\xi) > 1$ then $\phi(g(\xi)) < 0$.
\label{le:Zexists}
\end{lemma}

\begin{figure}[b!]
\begin{center}
\includegraphics[width=8cm]{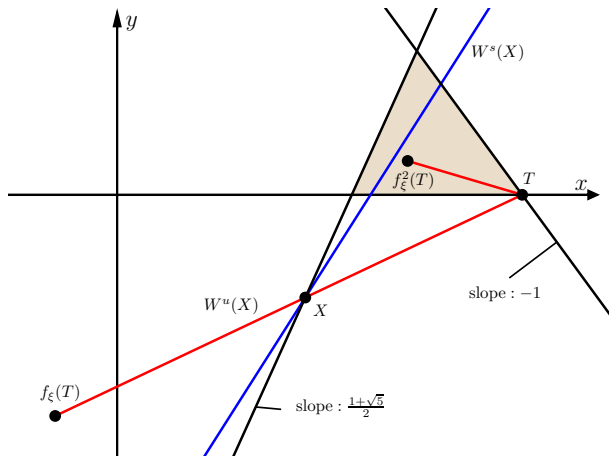}
\caption{
A sketch of the phase space of \eqref{eq:f} with $\xi \in \Phi$ and $\phi(g(\xi)) \ge 0$
to support the proof of Lemma \ref{le:Zexists}.
The shaded region is bounded by the $x$-axis,
the line through $X$ with slope $\frac{1 + \sqrt{5}}{2}$,
and the line through $T$ with slope $-1$.
\label{fig:zd_Proof}
} 
\end{center}
\end{figure}

\begin{proof}
Let $\xi \in \Phi$ and suppose $\phi(g(\xi)) \ge 0$.
It remains to show $J_1(\xi) \le 1$.
Let $T \in f_\xi(\Sigma)$ denote the first kink of $W^u(X)$ as we follow it outwards from $X$, see Fig.~\ref{fig:zd_Proof}.
As shown in \cite{GhSi21},
$\phi(g(\xi)) \ge 0$ implies $f_\xi^2(T)$ lies on or to the right of the stable subspace $E^s(X)$.
The slope of $E^s(X)$ is $\left| \lambda_R^u \right|$,
where $\left| \lambda_R^u \right| < \frac{1 + \sqrt{5}}{2}$ because $\phi(g(\xi)) \ge 0$, see \cite{GhSi21}.

The slope of $\lineSeg{T}{f_\xi(T)}$ is $|\lambda_R^s|$.
This slope belongs to the interval $K = \left[ -\lambda_L^s, |\lambda_R^s| \right]$ of Lemma \ref{le:K},
thus by the invariance of $\Psi_K$ the slope of
$\lineSeg{T}{f_\xi^2(T)}$ also belongs to this interval.
Thus the slope of $\lineSeg{T}{f_\xi^2(T)}$ is at least $-\lambda_L^s$, which is greater than $-1$.
Also $f_\xi^2(T)$ lies above $f_\xi(\Sigma)$ (the $x$-axis) because $f_\xi(T)$ lies to the left of $\Sigma$.
By putting these observations together we conclude that
$f_\xi^2(T)$ must belong to the shaded region of Fig.~\ref{fig:zd_Proof}.
It is then a simple exercise in geometry to show that
the distance of any point in this region to $T$
is less than the distance of $X$ to $T$.
Thus $f_\xi^2(T)$ is closer to $T$ than $X$ is.

Now let $\alpha = \lineSeg{X}{f_\xi(T)}$.
Then $f_\xi(\alpha) = \lineSeg{X}{T} \cup \lineSeg{T}{f_\xi^2(T)}$.
We have just shown that the longest line segment in $f_\xi(\alpha)$ is $\lineSeg{X}{T}$.
By Lemma \ref{le:longestLength} this has length
\begin{equation}
s \ge \frac{\lambda_L^u \left| \lambda_R^u \right|}{\lambda_L^u + \left| \lambda_R^u \right|} \,|\alpha|,
\label{eq:ZexistsProof}
\end{equation}
using also Lemma \ref{le:K} with $c_R = \left| \lambda_R^u \right|$
because $\left| \lambda_R^u \right| < \frac{1 + \sqrt{5}}{2}$.
Since $\lineSeg{X}{T}$ is aligned with the unstable direction of $X$,
$\left| f_\xi(\lineSeg{X}{T}) \right| = \left| \lambda_R^u \right| s$.
But $f_\xi(\lineSeg{X}{T}) = \alpha$, thus \eqref{eq:ZexistsProof} implies
$|\alpha| \ge \frac{\lambda_L^u \left| \lambda_R^u \right|^2}{\lambda_L^u + \left| \lambda_R^u \right|} \,|\alpha|$.
That is, $J_1(\xi) \le 1$.
\end{proof}

Next we prove that, under certain conditions, when we iterate a line segment
under $f_\xi$ we must eventually intersect $\lineSeg{V}{f_\xi^{-1}(V)}$ --- the initial linear piece of $W^s(X)$ as it emanates from $X$.
Then only a few more arguments are required to establish the denseness of $W^s(X)$.

\begin{lemma}
Let $\xi \in \Phi_{\rm BYG}$ and suppose $J_1(\xi) > 1$ and $\lambda_L^s + |\lambda_R^s| < 1$.
Let $\alpha \subset \Omega$ be a line segment with slope in $K = \left[ -\lambda_L^s, |\lambda_R^s| \right]$.
Then there exists $n \ge 1$ and points $P \in \Sigma$ and $Q \in f_\xi(\Sigma)$ such that $\lineSeg{P}{Q} \subset f_\xi^n(\alpha)$.
Moreover, $\lineSeg{P}{Q}$ intersects $\lineSeg{V}{f_\xi^{-1}(V)}$ transversally.
\label{le:intersectionWithBothAxes}
\end{lemma}

\begin{proof}
Let $\alpha_0 = \alpha$.
We construct a sequence $\{ \alpha_i \}$ of line segments in $\Omega$ with slopes in $K$ as follows.
If $\alpha_i$ does not cross $\Sigma$, let $\alpha_{i+1} = f_\xi(\alpha_i)$.
Observe $\alpha_{i+1} \subset \Omega$ because $\Omega$ is forward invariant under $f_\xi$ (Lemma \ref{le:fOmega}),
and $\alpha_{i+1}$ has slope in $K$ because $\Psi_K$ is invariant for both $A_L$ and $A_R$ (Lemma \ref{le:K}).
Moreover, with $c_R$ as in Lemma \ref{le:K},
$|\alpha_{i+1}| \ge d_0 |\alpha_i|$ where $d_0 = {\rm min} \left\{ \lambda_L^u, c_R \right\} > 1$.

If $\alpha_i$ crosses $\Sigma$, first let $\beta_i$ be the longer of the two line segments that comprise $f_\xi(\alpha_i)$.
If $\beta_i$ does not cross $\Sigma$, let $\alpha_{i+1} = f_\xi(\beta_i)$.
Both $\beta_i$ and $\alpha_{i+1}$ belong to $\Omega$ and have slopes in $K$ by the invariance properties.
We now show that the situation of $\beta_i$ not crossing $\Sigma$ cannot always occur.

The endpoint of $\beta_i$ that lies on $f_\xi(\Sigma)$ belongs the right half-plane.
This is because $\beta_i \subset f_\xi(\Omega)$ and the intersection of $f_\xi(\Omega)$ with $f_\xi(\Sigma)$
is the line segment with endpoints $f_\xi(U)$ and $D$, see Lemma \ref{le:fOmega}.
But $U$ lies above $V$ (Lemma \ref{le:U2V2}) because we have assumed $\lambda_L^s + |\lambda_R^s| < 1$,
thus $f_\xi(U)$ lies to the right of $f_\xi(V)$,
which certainly lies to the right of $\Sigma$ in view of the explicit expression \eqref{eq:V}.
So the assumption that $\beta_i$ does not cross $\Sigma$ means that $\beta_i$ lies in the closed right half-plane,
thus $|\alpha_{i+1}| \ge c_R |\beta_i|$ by Lemma \ref{le:K}.
Also $|\beta_i| \ge \frac{\lambda_L^u c_R}{\lambda_L^u + c_R} |\alpha_i|$ by Lemma \ref{le:longestLength},
therefore $|\alpha_{i+1}| \ge d_1 |\alpha_i|$ where
$d_1 = \frac{\lambda_L^u c_R^2}{\lambda_L^u + c_R}$.

If $|\lambda_R^u| > \frac{1 + \sqrt{5}}{2}$ then $c_R > \frac{1 + \sqrt{5}}{2}$ and so $d_1 > 1$ (because $\lambda_L^u > 1$).
Otherwise $c_R = |\lambda_R^u|$ in which case $d_1 = J_1(\xi)$ and so again $d_1 > 1$ (because $J_1(\xi) > 1$ by assumption).
Now let $d = {\rm min} \{ d_0, d_1 \}$.
Then $|\alpha_n| \ge d^n |\alpha_0| \to \infty$ as $n \to \infty$ because $d > 1$,
but $\Omega$ is bounded so this is not possible.
Thus there must exist $j \ge 0$ such that both $\alpha_j$ and $\beta_j$ cross $\Sigma$.
Let $P = \beta_j \cap \Sigma$ and $Q$ be the endpoint of $\beta_j$ that lies on $f_\xi(\Sigma)$.
By construction $\lineSeg{P}{Q} \subset f_\xi^n(\alpha)$ for some $j + 1 \le n \le 2 j + 1$.

Finally, $\lineSeg{P}{Q} \subset f_\xi(\Omega)$
thus $P$ lies on or above $U$ while $Q$ lies on or to the right of $f_\xi(U)$, see Fig.~\ref{fig:zd_fOmega}.
We now use Lemma \ref{le:U2V2} to show that $\lineSeg{P}{Q}$ intersects $\lineSeg{V}{f_\xi^{-1}(V)}$ transversally.
Since $U$ lies above $V$, $P$ lies to the left of the stable subspace $E^s(X)$.
Since $f_\xi(U)$ lies to the right of $f_\xi(V)$, $Q$ lies to the right of $E^s(X)$.
Thus $\lineSeg{P}{Q}$ intersects $E^s(X)$ transversally at some point $S \in f_\xi(\Omega)$.
Finally $S \in \lineSeg{V}{f_\xi^{-1}(V)}$
because $f_\xi^{-1}(V)$ lies above the line through $f_\xi^{-1}(D)$ and $D$
(which is the case because $V$ lies below the line through $D$ and $f_\xi(D)$).
\end{proof}

\begin{proof}[Proof of Theorem \ref{th:stableManifoldDense}]
Since $X \in \Omega$, see \cite{GhSi21}, and $\Omega$ is compact and forward invariant under $f_\xi$,
we have $\Lambda \subset \Omega$.
Choose any $\tilde{P} \in \Omega$ and $\ee > 0$.
Let $\alpha \subset B_\ee(\tilde{P}) \cap \Omega$ be a line segment with slope in $\left[ -\lambda_L^s, |\lambda_R^s| \right]$.
By Lemma \ref{le:intersectionWithBothAxes} there exists $n \ge 1$ such that
$f_\xi^n(\alpha)$ intersects $W^s(X)$ at some point $S$.
Thus $\tilde{Q} = f_\xi^{-n}(S)$
belongs to $W^s(X)$ and lies within a distance $\ee$ of $\tilde{P}$.
Since $\tilde{P} \in \Omega$ and $\ee > 0$ are arbitrary,
we can conclude that $W^s(X)$ is dense in $\Omega$.
\end{proof}

\section{Dynamics of the inverse $f_\xi^{-1}$}
\label{sec:fInverse}
\setcounter{equation}{0}

In this section we identify invariant expanding cones for $A_L^{-1}$ and $A_R^{-1}$.
Similar calculations were done in \cite{Mi80} to verify uniform hyperbolicity for the Lozi map.
The results, Lemmas \ref{le:ALinverse} and \ref{le:ARinverse}, are analogous to the results obtained above in
\S\ref{sec:invariantCone} for $A_L$ and $A_R$ and can be proved in the same way.
We omit such proofs for brevity but instead provide a novel proof of Lemma \ref{le:ALinverse}
that works by converting $A_L^{-1}$, via a change of coordinates, to a matrix that has the same companion matrix form as $A_L$.
We also show that the set of points whose backward orbits under $f_\xi$ diverge is dense in $\mathbb{R}^2$,
Lemma \ref{le:XiDense}.
Consequently the attractor $\Lambda$ cannot contain open sets and this
observation is utilised in the next section.

Given an interval $K$, the set
\begin{equation}
\hat{\Psi}_K = \left\{ t \begin{bmatrix} \hat{m} \\ 1 \end{bmatrix} \,\middle|\, \hat{m} \in K, t \in \mathbb{R} \right\},
\label{eq:PsiHatK}
\end{equation}
is a cone.

\begin{lemma}
Suppose $\tau_L > \delta_L + 1$ and $\delta_L > 0$.
Let $K = \left[ -\frac{1}{\lambda_L^u}, \hat{m}_{\rm max} \right]$ for some $\hat{m}_{\rm max} \in [0,1]$.
The cone $\hat{\Psi}_K$ is invariant and expanding for $A_L^{-1}$ with expansion factor
$\hat{c}_L = {\rm min} \left\{ \frac{1}{\lambda_L^s}, \frac{1 + \lambda_L^s}{\sqrt{2} \lambda_L^s} \right\}$.
\label{le:ALinverse}
\end{lemma}

\begin{proof}
Let $P = \begin{bmatrix} 0 & 1 \\ 1 & 0 \end{bmatrix}$ and let
\begin{equation}
B_L = P A_L^{-1} P^{-1} = \begin{bmatrix} \frac{\tau_L}{\delta_L} & 1 \\ -\frac{1}{\delta_L} & 0 \end{bmatrix}.
\nonumber
\end{equation}
This matrix has the form of $A_L$ with eigenvalues $0 < \frac{1}{\lambda_L^u} < 1 < \frac{1}{\lambda_L^s}$
(the same eigenvalues as $A_L^{-1}$ by similarity).
By Lemma \ref{le:AL}\ref{it:AL2}, $\Psi_K$ is invariant and expanding for $B_L$ with expansion factor $\hat{c}_L$.
Multiplication by $P$ corresponds to a reflection about the line $y=x$,
therefore $\hat{\Psi}_K$ (the reflection of $\Psi_K$ about $y=x$) is invariant and expanding for $A_L^{-1}$ with same expansion factor.
\end{proof}

\begin{lemma}
Suppose $\tau_R < -\delta_R - 1$ and $\delta_R > 0$.
Let $K = \left[ \hat{m}_{\rm min}, \frac{1}{|\lambda_R^u|} \right]$ for some $\hat{m}_{\rm min} \in [-1,0]$.
The cone $\hat{\Psi}_K$ is invariant and expanding for $A_R^{-1}$ with expansion factor
$\hat{c}_R = {\rm min} \left\{ \frac{1}{|\lambda_R^s|}, \frac{1 + |\lambda_R^s|}{\sqrt{2} |\lambda_R^s|} \right\}$.
\label{le:ARinverse}
\end{lemma}

\begin{lemma}
Let $\xi \in \Phi$ and suppose $J_2(\xi) < 1$.
The set
\begin{equation}
\Xi = \left\{ (x,y) \in \mathbb{R}^2 \,\middle|\, \| f_\xi^{-i}(x,y) \| \to \infty ~\text{as $i \to \infty$} \right\}.
\label{eq:Xi}
\end{equation}
is dense in $\mathbb{R}^2$.
\label{le:XiDense}
\end{lemma}

\begin{proof}
Choose any $\tilde{P} \in \mathbb{R}^2$ and $\ee > 0$.
It remains for us to show $B_\ee(\tilde{P}) \cap \Xi \ne \varnothing$.

Let $\hat{K} = \left[ -\frac{1}{\lambda_L^u}, \frac{1}{|\lambda_R^u|} \right]$.
By Lemmas \ref{le:ALinverse} and \ref{le:ARinverse}, $\hat{K}$ is invariant for $A_L^{-1}$ and $A_R^{-1}$
with expansion factors $\hat{c}_L$ and $\hat{c}_R$ given above.
Notice $J_2(\xi) < 1$ is equivalent to
\begin{equation}
\frac{1}{\hat{c}_L} + \frac{1}{\hat{c}_R} < 1.
\label{eq:XiDenseProof10}
\end{equation}
Let $\alpha_0 \subset B_\ee(\tilde{P})$ be a line segment with slope in $\hat{K}$.
For each $i \ge 0$,
let $\alpha_{i+1} = f_\xi^{-1}(\alpha_i)$ if this is a line segment,
otherwise let $\alpha_{i+1}$ be the longest line segment in $f_\xi^{-1}(\alpha_i)$.
Each $\alpha_i$ has slope in $\hat{K}$ by invariance.
Analogous to \eqref{eq:longestLength}, we have $|\alpha_{i+1}| \ge d |\alpha_i|$ for all $i \ge 0$, where
\begin{equation}
d = \frac{\hat{c}_L \hat{c}_R}{\hat{c}_L + \hat{c}_R}.
\nonumber
\end{equation}
But $d > 1$ by \eqref{eq:XiDenseProof10}, thus $|\alpha_i| \to \infty$ as $i \to \infty$.
Thus there exists $\tilde{Q} \in \alpha_0$ with $\| f_\xi^{-i}(\tilde{Q}) \| \to \infty$ as $i \to \infty$,
so $\tilde{Q} \in B_\ee(\tilde{P}) \cap \Xi$.
\end{proof}

\section{Devaney chaos}
\label{sec:Devaney}
\setcounter{equation}{0}

In this section we work towards a proof of Theorem \ref{th:DevaneyChaos}.

\begin{figure}[b!]
\begin{center}
\includegraphics[width=8cm]{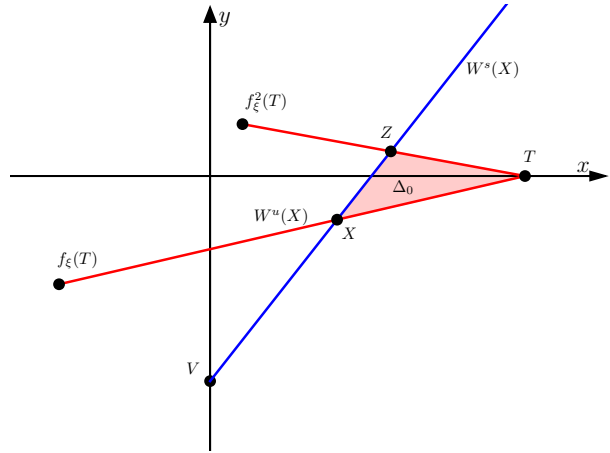}
\caption{
A sketch of the phase space of \eqref{eq:f} with $\xi \in \Phi$ and $\phi(g(\xi)) < 0$
illustrating the definition of $\Delta_0$ (shaded).
\label{fig:zd_Delta0}
} 
\end{center}
\end{figure}

As in the proof of Lemma \ref{le:Zexists}, let $T \in f_\xi(\Sigma)$ denote the first kink of $W^u(X)$.
But now suppose $f_\xi^2(T)$ lies to the left of the stable subspace $E^s(X)$.
As shown in \cite{GhSi21} this assumption is equivalent to $\phi(g(\xi)) < 0$,
so by Lemma \ref{le:Zexists} occurs when $J_1(\xi) > 1$.
Let $Z$ denote the intersection of $E^s(X)$ with $\lineSeg{T}{f_\xi^2(T)}$
and let $\Delta_0$ be the filled triangle with vertices $X$, $T$, and $Z$, see Fig.~\ref{fig:zd_Delta0}.
Also let
\begin{align}
\Delta &= \bigcup_{i=0}^\infty f_\xi^i(\Delta_0), &
\tilde{\Delta} &= \bigcap_{i=0}^\infty f_\xi^i(\Delta).
\end{align}
We first prove the following lemma that extends a result of \cite{GlSi21} to a wider range of parameter values.

\begin{lemma}
Let $\xi \in \Phi_{\rm BYG}$ and suppose $J_1(\xi) > 1$ and $J_2(\xi) < 1$.
Then $\Lambda = \tilde{\Delta}$.
\label{le:Lambda}
\end{lemma}

\begin{proof}
Here we write $\partial F$ for the boundary of a set $F \subset \mathbb{R}^2$.

By definition, $\partial \Delta_0 \subset \lineSeg{X}{Z} \cup W^u(X)$.
Consequently $\partial f_\xi^i(\Delta_0) \subset \lineSeg{X}{f_\xi^i(Z)} \cup W^u(X)$ for all $i \ge 0$.
Thus $\partial \Delta \subset \lineSeg{Z}{f_\xi(Z)} \cup W^u(X)$
and so $\partial f_\xi^i(\Delta) \subset \lineSeg{f_\xi^i(Z)}{f_\xi^{i+1}(Z)} \cup W^u(X)$ for all $i \ge 0$.
Therefore $\partial \tilde{\Delta} \subset \Lambda$ because $f_\xi^i(Z) \to X$.
In view of its definition, $\tilde{\Delta}$ is backwards invariant under $f_\xi$.
Thus $\tilde{\Delta} \cap \Xi = \varnothing$, where $\Xi$ is the set \eqref{eq:Xi}.
But $\Xi$ is dense in $\mathbb{R}^2$ (Lemma \ref{le:XiDense}),
thus $\tilde{\Delta} = \partial \tilde{\Delta}$.
Hence $\tilde{\Delta} \subset \Lambda$.

To prove $\Lambda \subset \tilde{\Delta}$, choose any $P \in \Lambda$.
Let $\{ P^{(k)} \}$ be a sequence of points in $W^u(X)$ with $P^{(k)} \to P$ as $k \to \infty$.
For each $k$ we have $f_\xi^i(P^{(k)}) \to X$ as $i \to -\infty$,
so there exists $i_k \le 0$ such that $f_\xi^{i_k}(P^{(k)}) \in \lineSeg{X}{T} \subset \Delta_0$.
Thus $P^{(k)} \in \Delta$ for all $k$, so $P \in \Delta$.
Thus $\Lambda \subset \Delta$.
But $\Lambda$ is forward invariant under $f_\xi$, thus $\Lambda \subset \tilde{\Delta}$.
\end{proof}

\begin{proof}[Proof of Theorem \ref{th:DevaneyChaos}]
Observe $W^u(X) = \bigcup_{i \ge 0} f_\xi^i \left( \lineSeg{X}{T} \right) \setminus \{ X \}$.
The line segment $\lineSeg{X}{T}$ has slope $|\lambda_R^s|$ which belongs to the interval $K = \left[ -\lambda_L^s, |\lambda_R^s| \right]$.
By Lemma \ref{le:K} the cone $\Psi_K$ is invariant for both $A_L$ and $A_R$,
thus each $f_\xi^i \left( \lineSeg{X}{T} \right)$ is a union of line segments with slopes in $K$.

Thus for any $\tilde{P} \in \Lambda$ and $\ee > 0$,
there exists a line segment $\alpha \subset W^u(X) \cap B_\ee(\tilde{P})$ with slope in $K$.
By Lemma \ref{le:intersectionWithBothAxes}, there exists $n_1 \ge 1$
such that $f_\xi^{n_1}(\alpha)$ transversally intersects $W^s(X)$ at some point $S$.
Arbitrarily close to $S$ there exists a non-wandering set associated with a Smale horseshoe \cite{Sm67}.
In the non-wandering set periodic points of $f_\xi$ are dense.
Thus $f_\xi$ has a periodic point $P_{\rm per} \in B_\ee(\tilde{P})$ with $P_{\rm per} \in \Lambda$
because $W^u(X)$ is also dense in the non-wandering set.
This shows that periodic points of $f_\xi$ are dense in $\Lambda$.

Given any $\tilde{Q} \in \Lambda$, the Lambda Lemma \cite{AlSa97,PaDe82} implies
there exists a point $M \in f_\xi^{n_1}(\alpha)$ such that the forward orbit of $M$ under $f_\xi$ eventually enters $B_\ee(\tilde{Q})$.
Since $f_\xi^{-n_1}(M) \in B_\ee(\tilde{P})$, this shows that $f_\xi$ is transitive on $\Lambda$.
Lastly, $f_\xi$ exhibits sensitive dependence on $\Lambda$ by the result of \cite{BaBr92}.
\end{proof}

\section{Discussion}
\label{sec:conc}
\setcounter{equation}{0}

Planar maps are useful for explaining the dynamical behaviour of a wide range of physical systems.
For instance they arise as stroboscopic maps of periodically forced one-degree-of-freedom oscillators.
Planar piecewise-linear maps are perhaps the simplest class of nonlinear planar maps
and provide a useful test-bed for exploring theoretical aspects of chaos.
They arise in applications as approximations near grazing bifurcations of oscillators with stick-slip friction
and explain how friction can induce chaotic dynamics in a robust fashion \cite{DiBu08}.

The results here are obtained through a series of geometric constructions.
In particular we used the condition $J_1(\xi) > 1$
to show that line segments grow sufficiently quickly when iterated under the map.
It remains to identify different and presumably more powerful methods
to verify the conjecture that $J_1(\xi) > 1$
can in fact be replaced by the weaker condition $\phi(g(\xi)) < 0$.

\section*{Acknowledgements}

The authors were supported by Marsden Fund contract MAU1809,
managed by Royal Society Te Ap\={a}rangi.

{\footnotesize

\begin{thebibliography}{10}

\bibitem{ZeSp12}
E.~Zeraoulia and J.C. Sprott.
\newblock {\em Robust Chaos and its Applications.}
\newblock World Scientific, Singapore, 2012.

\bibitem{Va10}
S.~van Strien.
\newblock One-parameter families of smooth interval maps: {D}ensity of
  hyperbolicity and robust chaos.
\newblock {\em Proc. Amer. Math. Soc.}, 138(12):4443--4446, 2010.

\bibitem{GoGo18}
A.S. Gonchenko, S.V. Gonchenko, A.O. Kazakov, and A.D. Kozlov.
\newblock Elements of contemporary theory of dynamical chaos: {A} tutorial.
  {P}art {I}. {P}seudohyperbolic attractors.
\newblock {\em Int. J. Bifurcation Chaos}, 28(11):1830036, 2018.

\bibitem{GoKa21}
S.~Gonchenko, A.~Kazakov, and D.~Turaev.
\newblock Wild pseudohyperbolic attractor in a four-dimensional {L}orenz
  system.
\newblock {\em Nonlinearity}, 34:2018--2047, 2021.

\bibitem{GuWi79}
J.~Guckenheimer and R.F. Williams.
\newblock Structural stability of {L}orenz attractors.
\newblock {\em Publ. Math. IHES}, 50:59--72, 1979.

\bibitem{Tu99}
W.~Tucker.
\newblock The {L}orenz attractor exists.
\newblock {\em C. R. Acad. Sci. Paris}, 328:1197--1202, 1999.

\bibitem{GlSi20b}
P.A. Glendinning and D.J.W. Simpson.
\newblock Robust chaos and the continuity of attractors.
\newblock {\em Trans. Math. Appl.}, 4(1):tnaa002, 2020.

\bibitem{AlPu17}
J.F. Alves, A.~Pumari\~{n}o, and E.~Vigil.
\newblock Statistical stability for multidimensional piecewise expanding maps.
\newblock {\em Proc. Amer. Math. Soc.}, 145(7):3057--3068, 2017.

\bibitem{BaYo98}
S.~Banerjee, J.A. Yorke, and C.~Grebogi.
\newblock Robust chaos.
\newblock {\em Phys. Rev. Lett.}, 80(14):3049--3052, 1998.

\bibitem{DiBu08}
M.~di~Bernardo, C.J. Budd, A.R. Champneys, and P.~Kowalczyk.
\newblock {\em Piecewise-smooth Dynamical Systems. Theory and Applications.}
\newblock Springer-Verlag, New York, 2008.

\bibitem{Si16}
D.J.W. Simpson.
\newblock Border-collision bifurcations in $\mathbb{R}^n$.
\newblock {\em SIAM Rev.}, 58(2):177--226, 2016.

\bibitem{ZhMo08b}
Z.T. Zhusubaliyev and E.~Mosekilde.
\newblock Equilibrium-torus bifurcation in nonsmooth systems.
\newblock {\em Phys. D}, 237:930--936, 2008.

\bibitem{SzOs08}
R.~Szalai and H.M. Osinga.
\newblock Invariant polygons in systems with grazing-sliding.
\newblock {\em Chaos}, 18(2):023121, 2008.

\bibitem{NuYo92}
H.E. Nusse and J.A. Yorke.
\newblock Border-collision bifurcations including ``period two to period
  three'' for piecewise smooth systems.
\newblock {\em Phys. D}, 57:39--57, 1992.

\bibitem{GhSi21}
I.~Ghosh and D.J.W. Simpson.
\newblock Renormalisation of the two-dimensional border-collision normal form.
\newblock \texttt{arXiv:2109.09242}, 2021.

\bibitem{GlSi21}
P.A. Glendinning and D.J.W. Simpson.
\newblock A constructive approach to robust chaos using invariant manifolds and
  expanding cones.
\newblock {\em Discrete Contin. Dyn. Syst.}, 41(7):3367--3387, 2021.

\bibitem{De89}
R.L. Devaney.
\newblock {\em An Introduction to Chaotic Dynamical Systems.}
\newblock Addison-Wesley, New York, 2nd edition, 1989.

\bibitem{Mi80}
M.~Misiurewicz.
\newblock Strange attractors for the {L}ozi mappings.
\newblock In R.G. Helleman, editor, {\em Nonlinear dynamics, Annals of the New
  York Academy of Sciences}, pages 348--358, New York, 1980. Wiley.

\bibitem{Sa99b}
E.A. Sataev.
\newblock Ergodic properties of the {B}elykh map.
\newblock {\em J. Math. Sci.}, 95(5):2564--2575, 1999.

\bibitem{HiKr18}
S.~Hittmeyer, B.~Krauskopf, H.M. Osinga, and K.~Shinohara.
\newblock Existence of blenders in a {H}\'{e}non-like family: geometric
  insights from invariant manifold computations.
\newblock {\em Nonlinearity}, 31(10):R239--R267, 2018.

\bibitem{BoDi05}
C.~Bonatti, L.J. D\'{\i}az, and M.~Viana.
\newblock {\em Dynamics Beyond Uniform Hyperbolicity.}
\newblock Springer, New York, 2005.

\bibitem{BaBr92}
J.~Banks, J.~Brooks, G.~Cairns, G.~Davis, and P.~Stacey.
\newblock On {D}evaney's definition of chaos.
\newblock {\em Amer. Math. Monthly}, 99(4), 1992.

\bibitem{Si16b}
D.J.W. Simpson.
\newblock Unfolding homoclinic connections formed by corner intersections in
  piecewise-smooth maps.
\newblock {\em Chaos}, 26:073105, 2016.

\bibitem{PaTa93}
J.~Palis and F.~Takens.
\newblock {\em Hyperbolicity and sensitive chaotic dynamics at homoclinic
  bifurcations.}
\newblock Cambridge University Press, New York, 1993.

\bibitem{GrOt83}
C.~Grebogi, E.~Ott, and J.A. Yorke.
\newblock Crises, sudden changes in chaotic attractors, and transient chaos.
\newblock {\em Phys. D}, 7:181--200, 1983.

\bibitem{Os06}
H.M. Osinga.
\newblock Boundary crisis bifurcation in two parameters.
\newblock {\em J. Diff. Eq. Appl.}, 12(10):997--1008, 2006.

\bibitem{VeGl90}
D.~Veitch and P.~Glendinning.
\newblock Explicit renormalisation in piecewise linear bimodal maps.
\newblock {\em Phys. D}, 44:149--167, 1990.

\bibitem{Sm67}
S.~Smale.
\newblock Differentiable dynamical systems.
\newblock {\em Bull. Amer. Math. Soc.}, 73:747--817, 1967.

\bibitem{AlSa97}
K.T. Alligood, T.D. Sauer, and J.A. Yorke.
\newblock {\em Chaos. An Introduction to Dynamical Systems.}
\newblock Springer, New York, 1997.

\bibitem{PaDe82}
J.~Palis and W.~de~Melo.
\newblock {\em Geometric Theory of Dynamical Systems. {A}n Introduction.}
\newblock Springer-Verlag, New York, 1982.

\end{thebibliography}

}

\end{document}